\documentclass[reqno,dvipdfmx]{amsart}
\usepackage{amsthm,amsmath,amssymb,amsthm}
\usepackage[left=3.5cm, right=3.5cm, top=3cm, bottom=3cm, includefoot]{geometry}
\usepackage[all]{xy}
\usepackage{color}
\usepackage{appendix}
\usepackage{tikz}

\newtheoremstyle{mystyle}
  {}
  {}
  {}
  {}
  {\bfseries}
  {.}
  { }
  {}

\theoremstyle{mystyle}
\theoremstyle{definition}
\newtheorem{theorem}{Theorem}[section]
\newtheorem{proposition}{Proposition}[section]
\newtheorem{lemma}{Lemma}[section]
\newtheorem{corollary}{Corollary}[section]

\newtheorem{definition}{Definition}[section]
\newtheorem{remark}{Remark}[section]
\numberwithin{equation}{section}

\begin{document}
\title[A Choquet--Deny approach to quantum de Finetti theorem]{Variations on quantum de Finetti theorems and operator valued Martin boundaries: a Choquet--Deny approach}

\begin{abstract}
    We revisit the quantum de Finetti theorem. We state and prove a couple of variants thereof. In parallel, we introduce an operator version of the Martin boundary on quantum groups and prove generalizations of Biane's theorem. Our proof of the de Finetti theorem is new in the sense that it is based on an analogy with the theory of operator-valued Martin boundary that we introduce.
\end{abstract}

\author[B. Collins]{Beno\^{i}t Collins}
\address[B. Collins]{Department of Mathematics, Kyoto University, Kitashirakawa, Oiwake-cho, Sakyo-ku, Kyoto, 606-8502, Japan}
\email{collins@math.kyoto-u.ac.jp}

\author[T. Giordano]{Thierry Giordano}
\address[T. Giordano]{Department of Mathematics and Statistics, University of Ottawa, Ottawa, Ontario, Canada}
\email{giordano@uottawa.ca}

\author[R. Sato]{Ryosuke Sato}
\address[R. Sato]{Department of Mathematics, Hokkaido University, Kita 10, Kita-ku, Sapporo, Hokkaido, 060-0810, Japan}
\email{r.sato@math.sci.hokudai.ac.jp}

\maketitle

\allowdisplaybreaks{
\section{Introduction}
The quantum de Finetti theorem is a fundamental theorem in quantum information theory. It states that the collection of separable states corresponds precisely to the collection of infinitely symmetrically extendable states.
It was developed independently in operator algebras \cite{Stormer} and in quantum information theory \cite{KoenigRenner}.
This remains one of the most powerful tools for detecting bipartite entanglement.

The classical de Finetti theorem is a cornerstone of probability theory, asserting that an infinitely exchangeable sequence of random variables is a convex mixture of independent and identically distributed (i.i.d.) random variables. This profound result has found its analogues and generalizations in various mathematical fields. Parallel to this, the theory of random walks on classical groups features the concept of a Martin boundary, which provides a geometric and probabilistic framework for characterizing positive harmonic functions---functions invariant under the transition operator of a random walk---and positive sub-harmonic functions. The Choquet--Deny theorem establishes that such harmonic functions can be uniquely represented as integrals over this boundary. Traditionally, these classical theories primarily concern scalar-valued functions. We refer the reader to \cite{Woess} as a general reference on this subject.

In the quantum realm, the quantum de Finetti theorem extends this notion to quantum states, asserting that quantum states of a bipartite quantum system that are symmetrically extendable to arbitrarily larger quantum systems can be approximated by convex combinations of tensor product states. Namely, such states are separable. This finite version of the theorem is particularly vital in quantum information theory, serving as a powerful tool for entanglement detection: a state is entangled if it fails to be $l$-extendable for sufficiently large $l$. Here, $l$ presents the size of extended quantum systems. For the definition and a bit of context, we refer to Definition \ref{def:extendable}.


In this paper, we propose a new approach to  (quantum) de Finetti-type theorems, inspired by the theory of random walks on commutative groups, and, more precisely, on the notion of the Martin boundary and the Choquet-Deny theory.

While conventional proofs of quantum de Finetti theorems often rely on the representation theory of symmetric and unitary groups, our work offers a novel approach by drawing inspiration from non-commutative harmonic analysis and the Martin boundary theory for quantum random walks, pioneered by Biane \cite{Biane92,Biane92-2, Biane08} and Izumi \cite{Izumi}. 

In these theories, the harmonic elements, which are considered, are scalar-valued functions because studying more general functions do not enrich the theory. In their non-commutative generalizations, Biane in \cite{Biane08} and Izumi in \cite{Izumi} did not consider operator-valued boundaries. In this paper, we study operator-valued boundaries and observe that it is closely related to the problem of identifying separable states.

Specifically, we introduce matrix-valued Martin boundaries and sub-harmonic elements, which generalize their scalar-valued counterparts and are shown to be intimately linked to the concept of separable quantum states of a bipartite quantum system. Our notion of $k$-sub-extendability, central to our generalized quantum de Finetti theorem (Theorem 3.2), is directly inspired by these non-commutative sub-harmonic elements. This paper thus establishes a new connection between quantum de Finetti theorems and the operator-valued Martin boundary theory.

After a short section on notations, we provide in Section \ref{sec:qdFthm} a new proof of the generalized quantum de Finetti theorem, inspired by harmonic analysis. It mimics Biane's proof of the non-commutative Choquet-Deny theorem. See Theorem \ref{thm:qD}.

In Section \ref{sec:BCD}, which seems a priori unrelated to the preceding section, we study the matrix-valued Martin boundary on the dual of a compact (quantum) group. This extends the previous work of Biane \cite{Biane92}. Then we observe a phenomenon comparable to the one in Section 3: {\it the matrix-valued Martin boundary is the separable version of its non-matrix-valued counterpart}.

In Section 5, we explore further the relation between sections \ref{sec:qdFthm} and \ref{sec:BCD}  by showing an inclusion between the matrix-valued non-commutative Martin boundary and the collection of extendable matrix-valued states (see Proposition \ref{prop:subharmonic_vs_extensions}). This inclusion is based on a matrix-valued extension of a non-commutative Martingale convergence theorem considered by Biane and Izumi.

The last section delves into the problem of understanding the notion of exponentials better. We extend some results of Biane \cite{Biane92-2} to more groups, and in the matrix-valued setup. See Lemma \ref{lem:hilbert}. This allows us to identify cases when the injection described in part 3 is a bijection.

\medskip
B. C. is supported by JSPS Grant-in-Aid
Scientific Research (B) no. 21H00987,
and Challenging Research (Exploratory) no. 20K20882 and 23K17299.

T. G. is supported in part by a grant from NSERC, Canada.

R. S. was funded by JSPS Research Fellowship for Young Scientists PD (KAKENHI Grand Number 22J00573).

\section{Notations}\label{sec:notations}
In this section, we summarize the notations and conventions used throughout this paper.
Let $\mathbb{M}_n(\mathbb{C})$ denote the $*$-algebra of $n\times n$ matrices over $\mathbb{C}$ and $\mathrm{Tr}_n$ the (non-normalized) trace on $\mathbb{M}_n(\mathbb{C})$. We denote by $\mathbb{M}_n(\mathbb{C})_+$ the set of all positive semi-definite matrices. In the paper, elements in $\mathbb{M}_n(\mathbb{C})_+$ are simply called \emph{positive}. More generally, let $B(\mathcal{H})$ denote the $*$-algebra of bounded linear operators on a complex Hilbert space $\mathcal{H}$. Here, its multiplication is given as compositions of linear operators and its $*$-operation is given by taking adjoint operators. Let $\mathcal{A}\subseteq B(\mathcal{H})$ be a \emph{unital $C^*$-algebra}, i.e., $\mathcal{A}$ is a unital $*$-subalgebra of $B(\mathcal{H})$ and closed with respect to the topology induced by the operator norm. We denote by $\mathcal{A}_+$ the set of all positive elements in $\mathcal{A}$, that is, by definition, $x\in \mathcal{A}_+$ if and only if $x$ is self-adjoint and the spectrum of $x$ is contained in $\mathbb{R}_{\geq0}$. For any pair $x, y\in \mathcal{A}_+$ we write $x\leq y$ if $y-x\in \mathcal{A}_+$. We remark that $\mathcal{A}_+$ is a proper convex cone and that as $\mathcal{A}_+-\mathcal{A}_+=\mathcal{A}_\mathrm{sa}$ and $\mathcal{A}_+\cap (-\mathcal{A}_+)=\{0\}$, the self-adjoint elements $\mathcal{A}_\mathrm{sa}$ of $\mathcal{A}$ becomes a partially ordered real vector space by defining $x\leq y$ if $y-x\in \mathcal{A}_+$.

Let $\mathcal{A}^*$ denote the set of all (continuous) linear functionals on $\mathcal{A}$. We say that $\rho\in \mathcal{A}^*$ is \emph{positive} if $\rho(x)\geq0$ for any $x\in \mathcal{A}_+$, and $\mathcal{A}^*_+$ denotes the set of all positive linear functionals on $\mathcal{A}$. In addition, if $\rho(1_\mathcal{A})=1$ holds, then $\rho$ is said to be a \emph{state} of $\mathcal{A}$. Here, $1_\mathcal{A}$ is the multiplicative unit in $\mathcal{A}$. We say that $\rho\in \mathcal{A}^*_+$ is \emph{faithful} if $\rho(x)>0$ for any $x\in \mathcal{A}_+\backslash\{0\}$.

If $\mathcal{H}$ is finite-dimensional, then $B(\mathcal{H})$ coincides with $\mathrm{End}(\mathcal{H})$ (and $\mathbb{M}_n(\mathbb{C})$ if $\dim \mathcal{H}=n$). Namely, all linear operators on $\mathcal{H}$ are bounded. Moreover, the trace $\mathrm{Tr}$ on $B(\mathcal{H})$ can be defined. Using $\mathrm{Tr}$, we define the \emph{Hilbert--Schmidt inner product} on $B(\mathcal{H})$ by $\langle A, B\rangle:=\mathrm{Tr}(B^*A)$ for all $A, B\in B(\mathcal{H})$. Moreover, by the Hilbert--Schmidt inner product, $B(\mathcal{H})$ can be naturally identified with its dual space $B(\mathcal{H})^*$.

We will use the terminology and results from convex analysis in \cite{Phelps:book}, more precisely from Chapter 13 of \cite{Phelps:book}. Let $K$ be a proper convex cone (i.e., $K\cap (-K)=\{0\}$) and $E:=K-K$ be the associated partially ordered real vector space. Recall that a \emph{ray} $\alpha$ of $K$ is a set of the form $\mathbb{R}x =\{rx\mid r\geq 0\}$, where $x\in K\backslash\{0\}$. A ray $\alpha$ is said to be an \emph{extreme ray} of $K$ if whenever $x\in \alpha$ and $x=\lambda y+(1-\lambda)z$ ($y, z\in K$, $0<\lambda<1$), then $y, z\in \alpha$. We denote by $\mathrm{exr}(K)$ the union of extreme rays of $K$. Then an element $x\in K$ is in $\mathrm{exr}(K)$ if and only if $y=rx$ for some $r\geq 0$ whenever $0\leq y\leq x$.

If $K$ is a closed convex set, a nonempty subset $C$ of $K$ is called a \emph{cap} of $K$ if $C$ is compact, convex, and $K\backslash C$ is also convex. Note (\cite[Prop 13.1]{Phelps:book}) that if $C$ is a cap of the closed convex cone and $x$ is an extreme point of $C$ then $x$ lies on an extreme ray of $K$.

Recall (\cite[page 1]{Phelps:book}) also that the following notion which we will first use in Proposition \ref{prop:int_rep}. Let $X$ be a nonempty compact subset of a locally convex space $E$ and $\mu$ a Borel probability measure on $X$. A point $x\in E$ is represented by $\mu$ if $f(x)=\int_X fd\mu$ for every continuous linear functional on $E$. Moreover, if $\mu_1$ and $\mu_2$ are nonnegative regular Borel measure on $X$, then $\mu_1\prec \mu_2$ if $\int_X fd\mu_1\leq \int_Xfd\mu_2$ for any convex function $f$ on $X$.

\section{The generalized quantum de Finetti theorem via Choquet theory}\label{sec:qdFthm}

Throughout this section, we fix two natural numbers $n, m$ and $\rho\in \mathbb{M}_n(\mathbb{C})^*_+$ that is faithful. The identity map on $\mathbb{M}_n(\mathbb{C})$ is simply denoted by $\mathrm{id}_n$. Let $S_l$ denote the symmetric group of degree $l\geq 1$, and it acts on $\mathbb{M}_n(\mathbb{C})^{\otimes l}$ by
\[\sigma\cdot(x_1\otimes \cdots \otimes x_l):=x_{\sigma(1)}\otimes \cdots \otimes x_{\sigma(l)} \quad (\sigma \in S_l,\, x_1, \dots, x_l\in \mathbb{M}_n(\mathbb{C})).\]
We then denote by $(\mathbb{M}_n(\mathbb{C})^{\otimes l})^{S_l}$ the $*$-subalgebra of fixed points under this action. Moreover, the convex cone of \emph{sub-martingale sequences of symmetric quantum states} is defined by 
\begin{equation}\label{eq:extensions}
    K_{m, \rho}:=\left\{(x_l)_{l=0}^\infty\in \prod_{l=0}^\infty (\mathbb{M}_m(\mathbb{C})\otimes (\mathbb{M}_n(\mathbb{C})^{\otimes l})^{S_l})_+\, \middle|\, (\mathrm{id}_m\otimes \mathrm{id}_n^{\otimes l}\otimes \rho)(x_{l+1})\leq x_l \text{ for all }l\geq0\right\},
\end{equation}
where $\mathbb{M}_n(\mathbb{C})^{\otimes 0}:=\mathbb{C}$ and $S_0$ is the trivial group. We equip $K_{m, \rho}$ with the topology of component-wise weak${}^*$ convergence, where $(\mathbb{M}_n(\mathbb{C})^{\otimes l})^{S_l}$ is identified with the subspace of the $S_l$-invariant elements in $(\mathbb{M}_n(\mathbb{C})^{\otimes l})^*$ by the Hilbert--Schmidt inner product. We also remark that $K_{m, \rho}$ is a convex cone by component-wise operations.

\begin{lemma}\label{lem:choquet_convex1}
    $K_{m, \rho}$ is the closed convex hull of $\mathrm{exr}(K_{m, \rho})$.
\end{lemma}
\begin{proof}
    For any $r>0$ we define
    \[C_r:=\{x=(x_l)_{l=0}^\infty\in K_{m, \rho}\mid \mathrm{Tr}_m(x_0)\leq r\}.\]
    For any $l\geq 1$ and $x\in C_r$ we have $(\mathrm{Tr}_m\otimes \rho^{\otimes l})(|x_l|)=(\mathrm{Tr}_m\otimes \rho^{\otimes l})(x_l)\leq \mathrm{Tr}_m(x_0)\leq r$. Thus, $C_r$ is compact, and hence $C_r$ is a cap of $K_{m, \rho}$ (see \cite[Proposition 13.2]{Phelps:book}). Moreover, $K_{m, \rho}=\bigcup_{r>0} C_r$ holds. Thus, the assertion follows from the Choquet theorem \cite[p.81]{Phelps:book}.
\end{proof}

We remark that the above proof showed that $K_{m, \rho}$ is a union of its caps. Thus, combining with results in \cite[Chapter 13]{Phelps:book}, we obtain the following representation theorem.

\begin{proposition}\label{prop:int_rep}
    Every $x\in K_{m, \rho}\backslash\{0\}$ can be represented by a Borel probability measure $M_x$ on $K_{m, \rho}$ such that $M_x$ is supported on any closed subset of $K_{m, \rho}$ which contains $\mathrm{exr}(K_{m, \rho})$.
\end{proposition}
\begin{proof}
    By the proof of Lemma \ref{lem:choquet_convex1}, $K_{m, \rho}$ is a union of its caps. Thus, every $x\in K_{m, \rho}$ is contained in an appropriate cap $C$ (see \cite[Theorem in 82p]{Phelps:book}). By the Choquet-Bishop-de Leeuw theorem applied to $C$ (see \cite[Section 4]{Phelps:book}), $x$ is represented by a Borel probability measure on $C$ which vanishes on the Baire subset of $C\backslash \mathrm{ex}(C)$. By maximality with respect to $\prec$, we may assume that this measure $M_x$ is supported by any closed subset of $C$ which contains $\mathrm{ex}(C)$. Therefore, by \cite[Proposition 13.1]{Phelps:book}, $M_x$ is supported by any closed subset of $K_{m, \rho}$ which contains $\mathrm{exr}(K_{m, \rho})$.
\end{proof}

We obtain the following description of $\mathrm{exr}(K_{m, \rho})$.
\begin{theorem}\label{thm:ext_pt1}
    $\mathrm{exr}(K_{m, \rho})$ is contained in $\{(a\otimes b^{\otimes l})_{l=0}^\infty\mid a\in \mathbb{M}_m(\mathbb{C})_+, b\in \mathbb{M}_n(\mathbb{C})_+, \rho(b)\leq 1\}$.
\end{theorem}
\begin{proof}
    Let us fix $x=(x_l)_{l=0}^\infty\in \mathrm{exr}(K_{m, \rho})$. For any $\nu\in \mathbb{M}_n(\mathbb{C})^*_+$ we define
    \[x_\nu=(x_{\nu, l})_{l=0}^\infty:=((\mathrm{id}_m\otimes \mathrm{id}_n^{\otimes l}\otimes \nu)(x_{l+1}))_{l=0}^\infty.\]
    Since $\mathrm{id}_m\otimes \mathrm{id}_n^{\otimes l}\otimes \nu$ is (completely) positive, we have $x_{\nu, l}\in (\mathbb{M}_m(\mathbb{C})\otimes (\mathbb{M}_n(\mathbb{C})^{\otimes l})^{S_l})_+$. Moreover, since $x_{\nu, l}$ is $S_l$-invariant, we have
    \[(\mathrm{id}_m\otimes \mathrm{id}_n^{\otimes l}\otimes \rho)(x_{\nu, l+1})=(\mathrm{id}_m\otimes \mathrm{id}_n^{\otimes l}\otimes \rho\otimes \nu)(x_{l+2}) \leq (\mathrm{id}_m\otimes \mathrm{id}_n^{\otimes l}\otimes \nu)(x_{l+1})=x_{\nu, l}.\]
    Thus, we have $x_\nu\in K_{m, \rho}$. If $\nu\leq \rho$, for any $l\geq 0$ we also have
    \[x_{\nu, l}=(\mathrm{id}_m\otimes \mathrm{id}_n^{\otimes l}\otimes \nu)(x_{l+1})\leq (\mathrm{id}_m\otimes \mathrm{id}_n\otimes \rho)(x_{l+1})\leq x_l.\]
    Thus, $x_\nu\leq x$ holds, and hence there exists $b(\nu)\geq 0$ such that $x_\nu=b(\nu)x$ since $x\in \mathrm{exr}(K_{m, \rho})$. We claim that the mapping from $\nu$ to $b(\nu)$ can be extended linearly to a functional on $\mathbb{M}_n(\mathbb{C})^*$. Let $\nu\in \mathbb{M}_n(\mathbb{C})^*_+$ (not necessarily $\nu \leq \rho$). Since $\rho$ is faithful on $\mathbb{M}_n(\mathbb{C})$, there exists $r>0$ such that $r\nu\leq \rho$. Thus, for some $b(r\nu)>0$, we have $x_{r\nu}=b(r\nu)x$ and $rx_\nu=x_{r\nu}=b(r\nu)x$. Namely, $b(r\nu)/r$ depends only on $\nu$, and hence $b(\nu):=rb(\nu/r)$ is well defined. Moreover, the mapping $b\colon \mathbb{M}_n(\mathbb{C})^*_+\to \mathbb{R}_{\geq 0}$ is additive and positive-homogeneous. By the Jordan decomposition (see e.g., \cite{Pederson}), every linear functional on $\mathbb{M}_n(\mathbb{C})$ uniquely decomposes into a linear combination of four positive linear functionals. Thus, the mapping $b$ can be extended to a linear functional on $\mathbb{M}_n(\mathbb{C})$, and $b$ is naturally in $\mathbb{M}_n(\mathbb{C})=(\mathbb{M}_n(\mathbb{C})^*)^*$. Moreover, we have $b\in \mathbb{M}_n(\mathbb{C})_+$, and $(\mathrm{id}_m\otimes \mathrm{id}_n^{\otimes l}\otimes \nu)(x_{l+1})=(\mathrm{id}_m\otimes \mathrm{id}_n^{\otimes l}\otimes \nu)(x_l\otimes b)$ for any $l\geq0$ and $\nu\in \mathbb{M}_n(\mathbb{C})^*$. It implies that $x_{l+1}=x_l\otimes b$, and hence, iterating over $l$, we have $x_l=x_0\otimes b^{\otimes l}$.
\end{proof}

Let us recall that $a\in (\mathbb{M}_m(\mathbb{C})\otimes \mathbb{M}_n(\mathbb{C}))_+$ is called \emph{separable} if $a$ is in the convex hull of $\mathbb{M}_m(\mathbb{C})_+\otimes \mathbb{M}_n(\mathbb{C})_+:=\{b\otimes c\mid b\in \mathbb{M}_m(\mathbb{C})_+, \, c\in \mathbb{M}_n(\mathbb{C})_+\}$. Based on the results in this section, we obtain a characterization of separable elements, which is a generalization of the quantum de Finetti theorem.
\begin{definition}\label{def:extendable}
    Let $l$ be a strictly positive integer, $\rho\in \mathbb{M}_n(\mathbb{C})^*_+$ and $a\in (\mathbb{M}_m(\mathbb{C})\otimes \mathbb{M}_n(\mathbb{C}))_+$. An element $b$ in $(\mathbb{M}_m(\mathbb{C})\otimes (\mathbb{M}_n(\mathbb{C})^{\otimes l})^{S_l})_+$ a \emph{$l$-sub-extension} of $a$ (with respect to $\rho$) if 
    \[(\mathrm{id}_m\otimes \mathrm{id}_n\otimes \rho^{\otimes (l-1)})(b)\leq a.\] 
    If $a$ has a $l$-sub-extension, then $a$ is said to be \emph{$l$-sub-extendable} (with respect to $\rho$).
    If the inequalities of this definition are replaced by equalities, the notion of sub-extension (resp. sub-extendable) becomes extension (resp. extendable). 
\end{definition}

This notion is quite important as a means of characterizing separable states, and detecting entanglement. 
We refer to \cite{KoenigRenner} for papers laying the foundations of such applications.

Here we state our first main result.
\begin{theorem}[Generalized Quantum de Finetti Theorem]\label{thm:qD}
    For every $a\in (\mathbb{M}_m(\mathbb{C})\otimes \mathbb{M}_n(\mathbb{C}))_+$, the following are equivalent:
    \begin{enumerate}
        \item $a$ is separable.
        \item There exists a faithful linear functional $\rho\in \mathbb{M}_n(\mathbb{C})_+^*$ such that $a$ is $l$-sub-extendable with respect to $\rho$ for every $l\geq1$.
        \item $a$ is $l$-sub-extendable with respect to every faithful linear functional on $\mathbb{M}_n(\mathbb{C})$ and for every integer $l\geq1$.
    \end{enumerate}
\end{theorem}
\begin{proof}
    Let $\rho, \nu\in \mathbb{M}_n(\mathbb{C})_+^*$ be faithful. Then, there exists $r>0$ such that $r\nu \leq \rho$. Thus, if $b$ is a $l$-sub-extension of $a$ with respect to $\rho$, then $r^{l-1}b$ is an $l$-sub-extension of $a$ with respect to $\nu$. Therefore, (2) and (3) are equivalent.

    To show that (1) and (2) are equivalent, let us first note that if $a=b\otimes c$ for some $b\in \mathbb{M}_m(\mathbb{C})_+$ and $c\in \mathbb{M}_n(\mathbb{C})_+$, then we have $a=(\mathrm{id}_m\otimes \mathrm{id}_n\otimes \rho^{\otimes (l-1)})(\rho(c)^{-(l-1)}b\otimes c^{\otimes l})$, i.e., $a$ is $l$-sub-extendable for all $l\geq 1$. Hence, if $a$ is separable, it is clearly $l$-sub-extendable for all $l\geq 1$, and consequently (1) implies (2).
    
    Conversely, we assume that there exist $l$-sub-extensions $x_l$ of $a$ for all $l\geq1$. For any $k\geq l$ we define
    \[x_{l, k}:=(\mathrm{id}_m\otimes \mathrm{id}_n^{\otimes l}\otimes \rho^{\otimes k-l})(x_k) \in (\mathbb{M}_m(\mathbb{C})\otimes (\mathbb{M}_n(\mathbb{C})^{\otimes l})^{S_l})_+,\]
    which is also $l$-sub-extension of $a$. Since $(\mathrm{Tr}_m\otimes \rho^{\otimes l})(|x_{l, k}|)=(\mathrm{Tr}_m\otimes \rho^{\otimes k})(x_{l, k})\leq (\mathrm{Tr}_m\otimes \rho)(a)$, the sequence $(x_{l, k})_{k\geq l}$ is contained in a compact subset of $(\mathbb{M}_m(\mathbb{C})\otimes (\mathbb{M}_n(\mathbb{C})^{\otimes l})^{S_l})_+$. Thus, there exists a convergence subsequence of $(x_{l, k})_{k\geq l}$, and $x_{l, \infty}$ denotes its limit. By applying this argument inductively, we obtain a sequence $(x_{l, \infty})_{l=2}^\infty$ of $l$-sub-extensions of $a$ such that $(\mathrm{id}_m\otimes \mathrm{id}_n^{\otimes l}\otimes\rho)(x_{l+1, \infty})\leq x_{l, \infty}$ for all $l\geq 2$. Namely, we have $(x_{l, \infty})_{l=0}^\infty \in K_{m, \rho}$, where $x_{1, \infty}=a$ and $x_{0, \infty}=(\mathrm{id}_m\otimes \rho)(a)$. By Lemma \ref{lem:choquet_convex1} and Theorem \ref{thm:ext_pt1}, $a$ can be approximated by convex combinations of $\mathbb{M}_m(\mathbb{C})_+\otimes \mathbb{M}_n(\mathbb{C})_+$, and hence $a$ is separable.
\end{proof}

\section{Matrix-valued Biane--Choquet--Deny theorem}\label{sec:BCD}
\subsection{$*$-Algebras from compact groups}\label{sec:alg_cpt_groups}
We fix a compact group $G$ and introduce three $*$-algebras $C(G)$, $\mathbb{C}[G]$, and $\mathcal{U}(G)$. The notations in this section follow the standard text \cite{NT:book} on compact quantum groups. Let $C(G)$ denote the $*$-algebra of continuous functions on $G$. We denote by $\mathbb{C}[G]\subset C(G)$ the linear span of all \emph{matrix coefficients} of finite-dimensional unitary representations. Here, by matrix coefficients, we mean continuous functions on $G$ given as $g\in G\mapsto \langle \pi(g)\xi, \eta\rangle$ for any (finite-dimensional) unitary representation $(\pi, \mathcal{H})$ of $G$ and two vectors $\xi, \eta\in \mathcal{H}$. By the definitions of tensor product and contragredient representations, $\mathbb{C}[G]$ becomes a $*$-subalgebra of $C(G)$. Moreover, $\mathbb{C}[G]$ is a Hopf $*$-algebra whose comultiplication $\Delta_G$, antipode $S_G$, and counit $\epsilon_G$ are given as follows:
\begin{itemize}
    \item the comultiplication $\Delta_G\colon \mathbb{C}[G]\to \mathbb{C}[G]\otimes \mathbb{C}[G]\subset C(G\times G)$ is the $*$-homomorphism given by $\Delta_G(f)(s, t):=f(st)$ for any $f\in \mathbb{C}[G]$ and $s,t\in G$
    \item the antipode $S_G\colon \mathbb{C}[G]\to \mathbb{C}[G]$ is the anti-homomorphism given by $S_G(f)(s):=f(s^{-1})$ for any $f\in \mathbb{C}[G]$ and $s\in G$
    \item the counit $\epsilon_G\colon \mathbb{C}[G]\to \mathbb{C}$ is the $*$-homomorphism given by $\epsilon_G(f):=f(e)$ for any $f\in \mathbb{C}[G]$, where $e\in G$ is the identity element.
\end{itemize}

For any $l\geq 1$ let us consider the (algebraic) dual space $\mathcal{U}(G^l):=(\mathbb{C}[G]^{\otimes l})^*$. We denote by $(\,\cdot\,,\,\cdot\,)$ the natural dual pairing between $\mathcal{U}(G^l)$ and $\mathbb{C}[G]^{\otimes l}$. If $l=1$, we simply write $\mathcal{U}(G)$. By the Hopf $*$-algebra structure of $\mathbb{C}[G]$, the linear space $\mathcal{U}(G)$ also becomes a unital $*$-algebra. More explicitly, the multiplication and the $*$-operation are given as
\[(xy, f):=(x\otimes y, \Delta_G(f)), \quad (x^*, f):=\overline{(x, S_G(f)^*)}\quad (x, y\in \mathcal{U}(G), \, f\in \mathbb{C}[G]).\]
Then, $\epsilon_G$ becomes a multiplicative unit in $\mathcal{U}(G)$. In general, $\mathcal{U}(G)$ is not a Hopf $*$-algebra. However, a comultiplication-like unital $*$-homomorphism $\hat\Delta_G\colon \mathcal{U}(G)\to \mathcal{U}(G^2)$ is defined by
\[(\hat\Delta_G(x), f\otimes g):=(x, fg) \quad (x\in \mathcal{U}(G), \, f, g\in \mathbb{C}[G]).\]
We remark that $\mathcal{U}(G)\otimes \mathcal{U}(G)$ can be naturally embedded into $\mathcal{U}(G^2)$, but $\mathcal{U}(G^2)$ is strictly larger than $\mathcal{U}(G)\otimes \mathcal{U}(G)$ in general. Similarly, two unital $*$-homomorphisms $\hat\Delta_G\otimes \mathrm{id}$ and $\mathrm{id}\otimes \hat\Delta_G$ from $\mathcal{U}(G^2)$ to $\mathcal{U}(G^3)$ are defined by
\[((\hat\Delta_G\otimes \mathrm{id})(x), f\otimes g\otimes h):=(x, fg\otimes h), \quad ((\mathrm{id}\otimes \hat\Delta_G)(x), f\otimes g\otimes h):=(x, f\otimes gh)\]
for any $x\in \mathcal{U}(G^2)$ and $f, g, h\in \mathbb{C}[G]$. By the associativity of the multiplication of $\mathbb{C}[G]$, we have $(\hat\Delta_G\otimes \mathrm{id})\circ \hat\Delta_G=(\mathrm{id}\otimes \hat\Delta_G)\circ \hat\Delta_G$, and this is denoted by $\hat\Delta_G^{(2)}$. Inductively, we also obtain a unital $*$-homomorphism $\hat\Delta_G^{(l)}\colon \mathcal{U}(G)\to \mathcal{U}(G^l)$ for any $l\geq 2$.

Let $(\pi, \mathcal{H})$ be a finite-dimensional unitary representation of $G$. For any $\xi, \eta\in \mathcal{H}$ we denote by $c^\pi_{\xi, \eta}$ the corresponding matrix coefficient, that is, $c^\pi_{\xi, \eta}(t):=\langle \pi(t)\xi, \eta\rangle$. For any $x\in \mathcal{U}(G)$ we define $\pi(x)\in B(\mathcal{H})$ as the unique linear operator satisfying $\langle \pi(x)\xi, \eta\rangle=(x, c^\pi_{\xi, \eta})$. Then, it is easy to show that the mapping $x\in \mathcal{U}(G)\mapsto \pi(x)\in B(\mathcal{H})$ is a unital $*$-homomorphism. We remark that $\pi(e_t)=\pi(t)$ for any $t\in G$, where $e_t\in\mathcal{U}(G)$ is the evaluation map at $t$, that is, $(e_t, f):=f(t)$ for any $f\in \mathbb{C}[G]$. By the definition of the $*$-operation on $\mathcal{U}(G)$, we have $e_t^*=e_{t^{-1}}$ for any $t\in G$. Thus, any $*$-representation of $\mathcal{U}(G)$ produces a unitary representation of $G$. Similarly, for any two unitary representations $(\pi, \mathcal{H})$ and $(\lambda, \mathcal{K})$ of $G$ we obtain a unital $*$-homomorphism $\pi\otimes \lambda\colon \mathcal{U}(G^2)\to B(\mathcal{H}\otimes \mathcal{K})$ by $\langle (\pi\otimes \lambda)(x)\xi_1\otimes \xi_2, \eta_1\otimes \eta_2\rangle = (x, c^\pi_{\xi_1, \eta_1}\otimes c^\lambda_{\xi_2, \eta_2})$ for any vectors $\xi_1, \eta_1\in \mathcal{H}$ and $\xi_2, \eta_2\in \mathcal{K}$. Namely, the representation of $\mathcal{U}(G)$ given as $(\pi\otimes \lambda)\circ \hat \Delta_G$ corresponds to the tensor product representation $(\pi\otimes \lambda, \mathcal{H}\otimes \mathcal{K})$ of $G$. In what follows, $(\pi\otimes \lambda)\circ\hat\Delta_G$ is simply denoted by $\pi\otimes \lambda$.

Let $\widehat G$ denote the set of all equivalence classes of irreducible representations of $G$, and it is always assumed to be countable. Throughout the paper, we fix representatives $((\pi_\alpha, \mathcal{H}_\alpha))_{\alpha\in \widehat G}$ of all irreducible representations. By the Peter--Weyl theorem, $\mathcal{U}(G)$ is isomorphic to $\prod_{\alpha\in \widehat G}B(\mathcal{H}_\alpha)$ as a $*$-algebra, and  the isomorphism is given by $\Pi_G(x):=(\pi_\alpha(x))_{\alpha\in \widehat G}$. In what follows, we freely identify $\mathcal{U}(G)$ and $\prod_{\alpha\in \widehat G}B(\mathcal{H}_\alpha)$.

\subsection{Matrix-valued Martin boundary on the dual of compact groups}
First, we mention the notion of quantum random walks or random walks on non-commutative spaces. See \cite{Biane08} and references therein for more details. From this standpoint, a pair consisting of a unital $*$-algebra and a state is interpreted as a non-commutative analog of a probability space. More precisely, a unital $*$-algebra serves as an algebra of random variables, while a state represents a linear functional obtained through expectation. Thus, a state is a non-commutative counterpart of a probability measure. Our results in this section concern the non-commutative probability space from $\mathcal{U}(G)$. See also the comment below \eqref{eq:subharmonic}.

We remark that $x\in \mathcal{U}(G)_+$ if and only if $\pi(x)\in B(\mathcal{H})_+$ for any finite-dimensional unitary representation $(\pi, \mathcal{H})$ of $G$. We also say that $\rho\in \mathcal{U}(G)^*$ is \emph{positive} if $\rho(x)\geq0$ for any $x\in \mathcal{U}(G)_+$. By the dual pairing between $\mathcal{U}(G)$ and $\mathbb{C}[G]$, every $f\in \mathbb{C}[G]$ is naturally identified with $\widehat f\in \mathcal{U}(G)^*$ given by $\widehat f(x):=(x, f)$ for any $x\in \mathcal{U}(G)$. We denote by $\mathcal{U}(G)_*$ the image of this identification, i.e., $\mathcal{U}(G)_*:=\{\widehat f\mid f\in\mathbb{C}[G]\}$ and $\mathcal{U}(G)_{*, +}:=\mathcal{U}(G)_*\cap \mathcal{U}(G)^*_+$. We will use the following notation:

\begin{definition}
    For any $\rho\in \mathcal{U}(G)_*$ we define $f_\rho\in \mathbb{C}[G]$ by $\widehat{f_\rho}=\rho$.
\end{definition}

For any finite-dimensional unitary representation $(\pi, \mathcal{K})$ of $G$ and any $\xi\in \mathcal{K}$, let $\omega_\xi$ denote the positive linear form on $B(\mathcal{K})$ given by $\omega_\xi(x):=\langle x\xi, \xi\rangle$ for any $x\in B(\mathcal{K})$. Then $\omega_\xi\circ\pi\in \mathcal{U}(G)_{*, +}$ and moreover any element $\rho\in \mathcal{U}(G)_{*, +}$ has this form. Let $\rho\in \mathcal{U}(G)_{*, +}$ be given. Note first that by Schur orthogonality relations, $\rho$ is finitely supported, i.e., $\rho|_{B(\mathcal{H}_\alpha)}=0$ holds for all but finitely many $\alpha\in \widehat G$. For $\rho|_{B(\mathcal{H}_\alpha)}$, let $(\sigma_\alpha, \mathcal{K}_\alpha)$ be the GNS representation of $B(\mathcal{H}_\alpha)$ and let $\xi_\alpha\in \mathcal{K}_\alpha$ such that $\rho|_{B(\mathcal{H}_\alpha)}(x)=\langle \sigma_\alpha(x)\xi_\alpha, \xi_\alpha\rangle$ for any $x\in B(\mathcal{H}_\alpha)$. Note that $\mathcal{K}_\alpha$ is finite-dimensional, and $\mathcal{K}_\alpha=\{0\}$ for all but finitely many $\alpha\in \widehat G$. Let $\sigma$ be the direct sum
\[\sigma:=\bigoplus_{\alpha\in \widehat G}\sigma_\alpha\colon \mathcal{U}(G)\to \bigoplus_{\alpha\in\widehat G}B(\mathcal{K}_\alpha).\]
It is a finite-dimensional representation, and we have $\rho(x)=\langle \sigma(x)\xi, \xi\rangle$ for any $x\in \mathcal{U}(G)$, where $\xi:=\sum_{\alpha\in\widehat G}\xi_\alpha$.

For any $\rho\in \mathcal{U}(G)_*$ we define a linear map $P_\rho\colon \mathcal{U}(G)\to \mathcal{U}(G)$ by ``$P_\rho:=(\mathrm{id}\otimes \rho)\circ \Delta_G$''. The precise definition is as follows:
\begin{definition}\label{def:4.2}
    For any $\rho\in \mathcal{U}(G)_*$ we define a linear map $P_\rho\colon \mathcal{U}(G)\to \mathcal{U}(G)$ by
    \[(P_\rho(x), f):=(\hat\Delta_G(x), f\otimes f_\rho)=(x, ff_\rho)\quad (x\in \mathcal{U}(G), \, f\in \mathbb{C}[G]).\]
    Its matrix-valued version is also defined as $P_{m, \rho}:=\mathrm{id}_m\otimes \rho$ on $\mathbb{M}_m(\mathbb{C})\otimes \mathcal{U}(G)$.
\end{definition}

For any $\rho, \nu\in \mathcal{U}(G)_*$ the \emph{convolution} $\rho*\nu\in \mathcal{U}(G)_*$ is defined as $\rho*\nu:=\rho\circ P_\nu$. It is easy to check $\rho*\nu=(\rho\otimes \nu)\circ\hat \Delta_G$, where $\rho\otimes \nu$ is a linear functional naturally defined on $\mathcal{U}(G^2)$ for $\rho, \nu\in\mathcal{U}(G)_*$. Since $\mathbb{C}[G]$ is commutative, we have $\rho*\nu=\nu*\rho$, and equivalently, $P_\rho\circ P_\nu=P_\nu\circ P_\rho$ holds. We also define $\rho^{*l}:=\rho\circ P_\rho^{l-1}$ for any $l\geq 1$ and $\rho^{*0}:=\epsilon_G$.

\begin{proposition}\label{prop:cp}
    For any $\rho\in \mathcal{U}(G)_{*, +}$ the linear map $P_\rho$ is completely positive, that is, $P_{m, \rho}$ is positivity-preserving for all $m\geq0$.
\end{proposition}
\begin{proof}
    We may assume that $\rho=\omega_\xi\circ\pi$ for some finite-dimensional unitary representation $(\pi, \mathcal{K})$ of $G$ and $\xi\in \mathcal{K}$. Let $x\in \mathbb{M}_m(\mathbb{C})\otimes \mathcal{U}(G)$ be positive. Our goal is to show $P_{m, \rho}(x)\geq 0$. For any finite-dimensional unitary representation $(\lambda, \mathcal{H})$ of $G$ we have
    \[(\mathrm{id}_m\otimes \lambda)(P_{m, \rho}(x))=(\mathrm{id}_m\otimes \mathrm{id}_{B(\mathcal{H})}\otimes \omega_\xi)(\mathrm{id}_m\otimes \lambda\otimes \pi)(x)\geq 0,\]
    that is, $P_{m, \rho}(x)\geq0$.
\end{proof}

We define
\begin{equation}\label{eq:subharmonic}
    B_{m, \rho}:=\{x\in (\mathbb{M}_m(\mathbb{C})\otimes \mathcal{U}(G))_+ \mid P_{m, \rho}(x)\leq x\}.
\end{equation}
In the viewpoint of noncommutative probability theory, $P_\rho$ is regarded as the transition operator of a Markov process. Thus, $B_{m, \rho}$ can be naturally interpreted as the space of $\mathbb{M}_m(\mathbb{C})$-valued $P_\rho$-subharmonic functions. In the paper, an element in $B_{m, \rho}$ is called \emph{$P_{m, \rho}$-subharmonic}. If $P_{m, \rho}(x)=x$, then $x$ is called \emph{$P_{m, \rho}$-harmonic}.

We equip $\mathbb{M}_m(\mathbb{C})\otimes \mathcal{U}(G)$ with the weak${}^*$-topology induced by $\mathbb{M}_m(\mathbb{C})\otimes \mathbb{C}[G]$. Here, $\mathbb{M}_m(\mathbb{C})$ is identified with its dual space by the Hilbert--Schmidt inner product. By $\mathrm{id}_m\otimes \Pi_G$, we have $\mathbb{M}_m(\mathbb{C})\otimes \mathcal{U}(G)\cong \prod_{\alpha \in \widehat G}\mathbb{M}_m(\mathbb{C})\otimes B(\mathcal{H}_\alpha)$, where the right-hand side is endowed with the product topology of the weak${}^*$ topologies, and this is equivalent to the topology of $\mathbb{M}_m(\mathbb{C})\otimes \mathcal{U}(G)$.

By a proof as in Lemma \ref{lem:choquet_convex1}, we have
\begin{lemma}\label{lem:choquet_convex2}
    $B_{m, \rho}$ is the closed convex hull of $\mathrm{exr}(B_{m, \rho})$.
\end{lemma}
\begin{proof}
    For any $c=(c_\alpha)_{\alpha\in \widehat G}\in \mathbb{R}_{>0}^{\widehat G}$ we define $\varphi_c\colon (\mathbb{M}_m(\mathbb{C})\otimes \mathcal{U}(G))_+\to [0, \infty]$ by
    \[\varphi_c(x)=\sum_{\alpha \in \widehat G}c_\alpha (\mathrm{Tr}_m\otimes \mathrm{Tr}_\alpha\circ \pi_\alpha)(x),\]
    where $\mathrm{Tr}_\alpha$ is the non-normalized trace on $B(\mathcal{H}_\alpha)$. Then, $\varphi_c$ is lower semi-continuous, additive, and positive-homogeneous. Moreover, $C_{\varphi_c}:=\{x\in B_{m, \rho}\mid \varphi_c(x)\leq 1\}$ is compact since
    \[C_{\varphi_c}\subset \{x\in \mathbb{M}_m(\mathbb{C})\otimes \mathcal{U}(G)\mid (\mathrm{Tr}_m\otimes \mathrm{Tr}_\alpha\circ\pi_\alpha)(|x|)<c_\alpha^{-1}\text{ for any }\alpha\in \widehat G\}.\]
    Thus, $C_{\varphi_c}$ is a cap of $B_{m, \rho}$ (see \cite[Proposition 13.2]{Phelps:book}). Thus, the claim follows from \cite[Chqoeut theorem in p.81]{Phelps:book} since $B_{m, \rho}=\bigcup_{c\in \mathbb{R}_{>0}^{\widehat G} }C_{\varphi_c}$.
\end{proof}

As in Lemma \ref{lem:choquet_convex1}, we showed that $B_{m, \rho}$ is a union of its caps in the above proof. Thus, by the same proof of Proposition \ref{prop:int_rep}, we obtain the following representation theorem:
\begin{proposition}\label{prop:int_rep2}
    Every $x\in B_{m, \rho}\backslash\{0\}$ can be represented by a Borel probability measure $M_x$ on $B_{m, \rho}$ such that $M_x$ is supported on any closed subset of $B_{m, \rho}$ which contains $\mathrm{exr}(B_{m, \rho})$.
\end{proposition}

We now study the extreme points $\mathrm{exr}(B_{m, \rho})$ of $B_{m, \rho}$ and obtain a matrix-valued extension of a theorem by Biane \cite{Biane92} that is itself a non-commutative version of the Choquet--Deny theorem.

Let us introduce the following notions. A non-zero element $e\in \mathcal{U}(G)$ is said to be \emph{group-like} (or \emph{primitive}) if $\hat\Delta_G(e)=e\otimes e$ holds. We denote by $\mathcal{G}(G)$ the subset of all group-like elements of $\mathcal{U}(G)$. We also define by $\mathcal{E}(G):=\mathcal{G}(G)\cap \mathcal{U}(G)_+$, the subset of all \emph{exponential} elements.
Then for any $e\in \mathcal{E}(G)$ and $f\in \mathbb{C}[G]$ we have that:
\[(P_\rho(e), f)=(\hat\Delta_G(e), f\otimes f_\rho)=\rho(e)(e, f).\]
Thus, an exponential element $e\in \mathcal{E}(G)$ is $P_\rho$-subharmonic if and only if $\rho(e)\leq 1$. If $\mathcal{E}_\rho(G)$ denotes the intersection $\mathcal{E}(G)\cap B_{1, \rho}$, we then have
\[\mathcal{E}_\rho(G)=\{e\in \mathcal{E}(G)\mid \rho(e)\leq 1\}.\]

We assume that $\rho\in \mathcal{U}(G)_{*, +}$ is \emph{generating}, i.e., for every $\nu\in \mathcal{U}(G)_{*, +}$ we have $\nu\leq \sum_{l=0}^ka_l \rho^{*l}$ for some $a_l\geq 0$ ($l=0, \dots, k$) and $k\geq0$.
\begin{theorem}[Matrix-valued Biane--Choquet--Deny Theorem]\label{thm:BCD}
    The set $\mathrm{exr}(B_{m, \rho})$ of extreme points of $B_{m, \rho}$ is contained in $\mathbb{M}_m(\mathbb{C})_+\otimes \mathcal{E}_\rho(G):=\{a\otimes e\mid a\in \mathbb{M}_m(\mathbb{C})_+,\, e\in \mathcal{E}_\rho(G)\}$. In particular, $B_{m, \rho}$ is the closed convex hull of $\mathbb{M}_m(\mathbb{C})_+\otimes \mathcal{E}_\rho(G)$.
\end{theorem}
\begin{proof}
    Let $x\in \mathrm{exr}(B_{m, \rho})$. For any $\nu\in \mathcal{U}(G)_{*, +}$ we have $P_{m, \nu}(x)\geq0$ by Proposition \ref{prop:cp}. Moreover, $P_{m, \rho}(P_{m, \nu}(x))=P_{m, \nu}(P_{m, \rho}(x))\leq P_{m, \nu}(x)$, and hence $P_{m, \nu}(x)\in B_{m, \rho}$. Since $\rho$ is generating, we have $\nu\leq \sum_{l=0}^ka_l\rho^{*l}$ for some $a_l\geq 0$ ($l=0, \dots, k$) and $k\geq0$. Thus, we have
    \[P_{m, \nu}(x)\leq \sum_{l=0}^ka_lP_\rho^l(x)\leq \left(\sum_{l=0}^ka_l\right)x,\]
    and hence there exists $\alpha(\nu)\geq 0$ such that $P_{m, \nu}(x)=\alpha(\nu)x$. By the Jordan decomposition (see e.g., \cite{Pederson}), any $\nu\in \mathcal{U}(G)_*$ is uniquely decomposed into a summation of four elements in $\mathcal{U}(G)_{*, +}$. Thus, $\alpha$ induces a linear functional on $\mathcal{U}(G)_*\cong \mathbb{C}[G]$. Therefore, there exists $e\in \mathcal{U}(G)_+$ such that $\nu(e)=\alpha(\nu)$ for any $\nu\in \mathcal{U}(G)$. Namely, we have $P_{m, \nu}(x)=(e, f_\nu)x$.

    For any $\mu\in \mathbb{M}_m(\mathbb{C})^*$, we define $x_\mu:=(\mu\otimes \mathrm{id}_{\mathcal{U}(G)})(x)\in \mathcal{U}(G)$. For any $f, g\in \mathbb{C}[G]$ we have
    \[(\hat\Delta_G(x_\mu), f\otimes g)=(P_{\widehat f}(x_\mu), g)=(e\otimes x_\mu, f\otimes g).\]
    Moreover, since $fg=gf$, we have
    \[(\hat\Delta_G(x_\mu), f\otimes g)=(P_{\widehat g}(x_\mu), f)=(x_\mu\otimes e, f\otimes g).\]
    Thus, $\hat\Delta_G(x_\mu)=x_\mu\otimes e=e\otimes x_\mu$ holds, and $e$ is group-like. Moreover, there exists $a\in \mathbb{M}_m(\mathbb{C})_+$ such that $x=a\otimes e$. Since $x$ is $P_{m, \rho}$-subharmonic, $e$ is also $P_\rho$-subharmonic, i.e., $e\in \mathcal{E}_\rho(G)$. Therefore, we have $\mathrm{exr}(B_{m, \rho})\subset \mathbb{M}_m(\mathbb{C})_+\otimes \mathcal{E}_\rho(G)$, and hence $B_{m, \rho}$ is contained in the closed convex hull of $\mathbb{M}_m(\mathbb{C})_+\otimes \mathcal{E}_\rho(G)$. On the other hand, since $\mathcal{E}_\rho(G)\subset B_{1, \rho}$ by definition, we have $\mathbb{M}_m(\mathbb{C})_+\otimes \mathcal{E}_\rho(G)\subset B_{m, \rho}$. Thus, the closed convex hull of $\mathbb{M}_m(\mathbb{C})_+\otimes \mathcal{E}_\rho(G)$ is contained in $B_{m, \rho}$, that is, they are equal.
\end{proof}

\begin{remark}
    Our results in this section can be extended to compact quantum groups with slight modifications. See, e.g., \cite{NT:book} for the basics of compact quantum groups. If $G$ is a compact \emph{quantum} group, then $\mathbb{C}[G]$ is non-commutative, and hence the convolution product on $\mathcal{U}(G)_*$ is also non-commutative. However, for any $\rho\in \mathcal{U}(G)_{*, +}$, Propositions \ref{prop:cp}, \ref{prop:int_rep2} and Lemma \ref{lem:choquet_convex2} hold true. In addition, if $\rho$ is generating and central (i.e., $\rho$ satisfies $\rho*\nu=\nu*\rho$ for any $\nu\in \mathcal{U}(G)_*$), then Theorem \ref{thm:BCD} holds true. However, we note that, in the case of $SU_q(2)$ for instance, every central element in $\mathcal{U}(SU_q(2))_*$ is proportional to the counit (i.e., it is not generating) since the center of $\mathbb{C}[SU_q(2)]$ is trivial (see \cite[Corollary 4.4]{KS}).
\end{remark}

By Theorem \ref{thm:BCD}, we can realize that $B_{m, \rho}$ is similar to the convex cones of separable elements. Namely, any subharmonic elements are produced from tensors of two positive elements. This observation and the generalized quantum de Finetti theorem (see Theorem \ref{thm:qD}) suggest the relation between $B_{m, \rho}$ and sub-extendability. We will investigate this point in the next section.

\section{The relation between the matrix-valued Martin boundary and the generalized quantum de Finetti theorem: quantum random walks}
In this part, we deepen our understanding of the relation between matrix-valued Martin boundary and the generalized quantum de Finetti theorem. Throughout this section, we fix a finite-dimensional unitary representation $(\Pi, \mathcal{H})$ of $G$ and $\rho\in B(\mathcal{H})_{*, +}$ that is faithful. Let us denote by the same symbol $\rho$ the linear functional on $\mathcal{U}(G)$ defined as $\rho\circ \Pi$. We remark that $\rho$ is generating if $(\Pi, \mathcal{H})$ is also \emph{generating} $\widehat G$, that is, for any $\alpha\in\widehat G$ there exists $l\geq 0$ such that $(\pi_\alpha, \mathcal{H}_\alpha)$ is a sub-representation of $(\Pi^{\otimes l}, \mathcal{H}^{\otimes l})$\footnote{For instance, the fundamental representation of $SU(n)$ is generating. On the other hand, it fails for $U(n)$. It is \emph{not} generating. In fact, the fundamental representation of $U(n)$ does not generate irreducible representations with highest weights containing negative parts. However, it plays the central role in the next section that concerns the relationship separable states of a bipartite quantum system and matrix-valued Martin boundary of random walks on the dual of $U(n)$.}.

Let us recall that the UHF (uniformly hyperfinite) algebra $B(\mathcal{H})^{\otimes\infty}$ is defined as the $C^*$-inductive limit of the $B(\mathcal{H})^{\otimes l}$ ($l\geq 0$). If $\rho$ is a state of $B(\mathcal{H})$, then the infinite tensor product state $\rho^{\otimes \infty}$ on $B(\mathcal{H})^{\otimes \infty}$ is defined. Thus, we obtain the $C^*$-probability space $(B(\mathcal{H})^{\otimes \infty}, \rho^{\otimes\infty})$. Moreover, we can regard the family of the subalgebras $B(\mathcal{H})^{\otimes l}$ as a \emph{non-commutative analogue} of filtration of $\sigma$-algebras. Namely, $(B(\mathcal{H})^{\otimes \infty}, (B(\mathcal{H})^{\otimes l})_{l=0}^\infty, \rho^{\otimes \infty})$ is a \emph{filtered} $C^*$-probability space. For every $l\geq 0$ the conditional expectation $E_l$ from $B(\mathcal{H})^{\otimes \infty}$ to $B(\mathcal{H})^{\otimes l}$ is also defined as the linear map induced by $\mathrm{id}_{B(\mathcal{H})}^{\otimes l}\otimes \rho^{\otimes k}\colon B(\mathcal{H})^{\otimes l+k}\to B(\mathcal{H})^{\otimes l}$ ($k\geq 0$). Under this point of view, a sequence $(x_l)_{l=0}^\infty$ in $(B(\mathcal{H})^{\otimes \infty}, \rho^{\otimes\infty})$ is said to be \emph{adapted sub-martingale} if $x_l\in B(\mathcal{H})^{\otimes l}$ and $E_l(x_{l+1})\leq x_l$ for every $l\geq0$. According to \eqref{eq:extensions}, we can interpret $K_{m, \rho}$ as the set of $\mathbb{M}_m(\mathbb{C})$-valued, positive, symmetric, adapted sub-martingale sequences. Here, $B(\mathcal{H})$ is naturally identified with $\mathbb{M}_n(\mathbb{C})$ if $\dim \mathcal{H}=n$. Such an interpretation was given initially by Izumi \cite{Izumi}. In his work, he initiated the non-commutative Poisson boundary theory. In the paper, our setting is slightly different from his work, but we can find a connection between such a non-commutative boundary theory and quantum information theory.

We compare $K_{m, \rho}$ and $B_{m, \rho}$. For any $l\geq 0$ we define
\[\Pi_{m, l}:=(\mathrm{id}_m\otimes \Pi^{\otimes l})\colon \mathbb{M}_m(\mathbb{C})\otimes \mathcal{U}(G)\to \mathbb{M}_m(\mathbb{C})\otimes B(\mathcal{H})^{\otimes l},\]
\[\Pi_{m, \infty}\colon \mathbb{M}_m(\mathbb{C})\otimes \mathcal{U}(G)\to \prod_{l=0}^\infty\mathbb{M}_m(\mathbb{C})\otimes B(\mathcal{H})^{\otimes l}\]
by $\Pi_{m, \infty}(x)=(\Pi_{m, l}(x))_{l=0}^\infty$, where $\Pi^{\otimes 0}:=\epsilon_G$. We remark that $\Pi_{m, l}(x)\in \mathbb{M}_m(\mathbb{C})\otimes (B(\mathcal{H})^{\otimes l})^{S_l}$ for any $x\in \mathbb{M}_m(\mathbb{C})\otimes \mathcal{U}(G)$ and $l\geq0$. In fact, we may observe the following: Let $\mu\in \mathbb{M}_m(\mathbb{C})^*$ and $x_\mu:=(\mu\otimes \mathrm{id}_{\mathcal{U}(G)})(x)\in \mathcal{U}(G)$. Since $\mathbb{C}[G]$ is a commutative algebra, for any $\sigma\in S_l$ and $\nu_1, \dots, \nu_l\in B(\mathcal{H})^*$ and $\mu\in \mathbb{M}_m(\mathbb{C})^*$ we have
\begin{align*}
    (\nu_1\otimes \cdots \otimes \nu_l)(\sigma\cdot \Pi_{1, l}(x_\mu))
     & =(\nu_{\sigma^{-1}(1)}\otimes\cdots\otimes \nu_{\sigma^{-1}(l)})(\Pi_{1, l}(x_\mu))                                     \\
     & =(\hat\Delta_G^{(l)}(x_\mu), f_{\nu_{\sigma^{-1}(1)}\circ \Pi}\otimes \cdots \otimes f_{\nu_{\sigma^{-1}(l)}\circ \Pi}) \\
     & =(x_\mu, f_{\nu_1\circ\Pi} \cdots f_{\nu_l\circ\Pi})                                                                    \\
     & =(\nu_1\otimes \cdots \otimes \nu_l)(\Pi_{1, l}(x_\mu)).
\end{align*}
Thus, we have $\sigma\cdot \Pi_{1, l}(x_\mu)=\Pi_{1, l}(x_\mu)$, and hence $\Pi_{m, l}(x)\in \mathbb{M}_m(\mathbb{C})\otimes (B(\mathcal{H})^{\otimes l})^{S_l}$.

We start to discuss the relation among the results of the previous two sections.
\begin{proposition}\label{prop:subharmonic_vs_extensions}
    $\Pi_{m, \infty}(B_{m, \rho})$ is contained in $K_{m, \rho}$. Moreover, $\Pi_{m, \infty}$ is injective if $(\Pi, \mathcal{H})$ is generating $\widehat G$.
\end{proposition}
\begin{proof}
    Ler $x\in B_{m, \rho}$. For any $l\geq 0$ we already know that $\Pi_{m, l}(x)$ is in $(\mathbb{M}_m(\mathbb{C})\otimes (B(\mathcal{H})^{\otimes l})^{S_l})_+$. Moreover, for any $\nu\in (B(\mathcal{H})^{\otimes l})^*_+$ we have
    \begin{align*}
        (\mathrm{id}_m\otimes \nu\otimes \rho)(\Pi_{m, l+1}(x))
         & =(\mathrm{id}_m\otimes ((\nu\circ\Pi^{\otimes l})*\rho))(x)     \\
         & =(\mathrm{id}_m\otimes \nu\circ\Pi^{\otimes l})(P_{m, \rho}(x)) \\
         & \leq (\mathrm{id}_m\otimes\nu\circ\Pi^{\otimes l})(x)           \\
         & =(\mathrm{id}_m\otimes \nu)(\Pi_{m, l}(x)).
    \end{align*}
    Therefore, $(\mathrm{id}_m\otimes \mathrm{id}_{B(\mathcal{H})}^{\otimes l} \otimes \rho)(\Pi_{m, l+1}(x))\leq \Pi_{m, l}(x)$ holds, and hence $\Pi_{m, \infty}(x)\in K_{m, \rho}$.

    We assume that $(\Pi, \mathcal{H})$ is generating $\widehat G$ and prove that $\Pi_{m, \infty}$ is injective. We suppose that $\Pi_{m, \infty}(x)=\Pi_{m, \infty}(y)$ for $x, y\in \mathbb{M}_m(\mathbb{C})\otimes \mathcal{U}(G)$. Since $(\Pi, \mathcal{H})$ is generating $\widehat G$, for any $\alpha\in \widehat G$ the corresponding irreducible representation $(\pi_\alpha, \mathcal{H}_\alpha)$ is contained in $(\Pi^{\otimes l}, \mathcal{H}^{\otimes l})$ for some $l\geq0$. Since $\Pi_{m, l}(x)=\Pi_{m, l}(y)$, we have $(\mathrm{id}_m\otimes \pi_\alpha)(x)=(\mathrm{id}_m\otimes \pi_\alpha)(y)$. Thus, $x=y$ holds, and hence $\Pi_{m, \infty}$ is injective.
\end{proof}

The next statement immediately follows from Theorem \ref{thm:BCD} if $(\Pi, \mathcal{H})$ is generating $\widehat G$. However, by Lemma \ref{lem:choquet_convex1}, Theorem \ref{thm:ext_pt1}, and the above proposition, the following holds true even if $(\Pi, \mathcal{H})$ is not generating $\widehat G$.

\begin{corollary}\label{cor:boundary_sep}
    $\Pi_{m, 1}(B_{m, \rho})$ is contained in the set of separable elements in $\mathbb{M}_m(\mathbb{C})\otimes \mathbb{M}_n(\mathbb{C})$.
\end{corollary}

\section{The case $G=U(n)$ and the generalized quantum de Finetti theorem revisited}
In this section, we consider the case $G=U(n)$. For any $t\in U(n)$, recall that $e_t\in \mathcal{U}(U(n))$ denotes the evaluation at $t$, i.e., $(e_t, f)=f(t)$ for all $f\in \mathbb{C}[U(n)]$. It is a group-like of $\mathcal{U}(U(n))$, as for any $f, g\in \mathbb{C}[U(n)]$
\[(\hat\Delta_{U(n)}(e_t), f\otimes g)=(e_t, fg)=f(t)g(t)=(e_t\otimes e_t, f\otimes g),\]

By Weyl's unitary trick, every finite-dimensional representation of $U(n)$ can be extended to a representation of $GL(n, \mathbb{C})$. Thus, for every $t\in GL(n, \mathbb{C})$ its evaluation map $e_t$ on $\mathbb{C}[U(n)]$ is well defined, and is also group-like.

The following was initially mentioned by Biane \cite[Theorem 5.1]{Biane92-2} when $G$ is a connected, simply connected semisimple compact Lie group. Using the same argument, we can then show for $G=U(n)$.

\begin{lemma}\label{lem:hilbert}
    The mapping $t\in GL(n, \mathbb{C})\mapsto e_t\in \mathcal{G}(U(n))$ is a bijection. Moreover, $e_t\in \mathcal{E}(U(n))$ if and only if $t\in GL(n, \mathbb{C})_+:=GL(n, \mathbb{C})\cap \mathbb{M}_n(\mathbb{C})_+$.
\end{lemma}
\begin{proof}
    Let $e\in\mathcal{G}(U(n))$. By Weyl's unitary trick, $e$ gives a homomorphism from $\mathbb{C}[GL(n, \mathbb{C})]$ to $\mathbb{C}$. We remark that $GL(n, \mathbb{C})$ is an algebraic set by the following identification:
    \[GL(n, \mathbb{C})\cong \{((x_{ij})_{i, j=1}^n, y)\in \mathbb{C}^{n^2+1}\mid y\det(x_{ij})_{i, j=1}^n-1=0\},\]
    where the identification map is given as $x\mapsto (x, \det(x)^{-1})$. Thus, by Hilbert's nullstellensatz, every maximal ideal of $\mathbb{C}[GL(n, \mathbb{C})]$ is the kernel of an evaluation map. In particular, $\ker(e)$ is a maximal ideal, and hence $\ker(e)=\ker(e_t)$ for some $t\in GL(n, \mathbb{C})$. Therefore, we have $e=e_t$, and the mapping $t\mapsto e_t$ is a bijection from $GL(n, \mathbb{C})$ to $\mathcal{G}(U(n))$.

    If $e_t$ is positive, then so is the fundamental representation of $e_t$, i.e., $t$ itself is positive. Conversely, if $t\in GL(n, \mathbb{C})_+$, then we consider its diagonalization $t=u\mathrm{diag}(a_1, \dots, a_n)u^*$, where $u\in U(n)$ and $\mathrm{diag}(a_1, \dots, a_n)$ is the diagonal matrix with entries $a_1, \dots, a_n\geq0$. For any unitary representation $(\pi, \mathcal{H})$ of $U(n)$,
    \[\pi(e_t)=\pi(t)=\pi(u)\pi(\mathrm{diag}(a_1, \dots, a_n))\pi(u)^*\]
    holds, where we remark that $\pi(u)$ is unitary. Moreover, it is well known that $\mathcal{H}$ is decomposed into a direct sum of root spaces. For any $(\lambda_1, \dots, \lambda_n)\in \mathbb{Z}^n$ the diagonal part $\pi(\mathrm{diag}(a_1, \dots, a_n))$ acts as the scalar $a_1^{\lambda_1}\cdots a_n^{\lambda_n}$ on the corresponding root space. Thus, $\pi(\mathrm{diag}(a_1, \dots, a_n))$ is positive, and hence so is $\pi(e_t)$. Namely, $e_t$ is exponential.
\end{proof}

\begin{remark}
    The description of group-like elements (i.e., the first statement of Lemma \ref{lem:hilbert}) is given in \cite[Theorem 3.2.2]{NT:book} for arbitrary compact Lie group.
\end{remark}

\begin{remark}
    An alternative proof of Lemma \ref{lem:hilbert} is also given as follows. Let us consider the embedding $U(n)\hookrightarrow SU(n+1)$ by $u\mapsto u\oplus \det(u)^{-1}$. Thus, we obtain the restriction map from $\mathbb{C}[SU(n+1)]$ to $\mathbb{C}[U(n)]$. Then, its dual map $\mathcal{U}(U(n))\to \mathcal{U}(SU(n+1))$ is an injective unital $*$-homomorphism intertwining both comultiplications. Thus, if $e$ is a group-like element in $\mathcal{U}(U(n))$, then so is in $\mathcal{U}(SU(n+1))$. By \cite[Theorem 5.1]{Biane92-2}, there exists $t\in SL(n+1, \mathbb{C})$ such that $e=e_t$. Let us consider the representation of $\mathcal{U}(SU(n+1))$ induced by the fundamental representation of $SU(n+1)$. Since $e\in \mathcal{U}(U(n))$, the image of $e$ (i.e., $t$) has the form $t'\oplus \det(t')^{-1}$ for some $t'\in GL(n, \mathbb{C})$. Therefore, $e=e_{t'}$ as elements in $\mathcal{U}(U(n))$. Moreover, by \cite[Theorem 5.1]{Biane92-2}, $e$ is positive if and only if $t\in SL(n+1, \mathbb{C})_+$, and it is equivalent to that $t'\in GL(n, \mathbb{C})_+$.
\end{remark}

In what follows, let $(\Pi, \mathbb{C}^n)$ denote the fundamental representation of $U(n)$ and fix a faithful positive linear functional $\rho$ on $\mathbb{M}_n(\mathbb{C})$ ($=B(\mathbb{C}^n)$). As in the previous section, $\rho$ will also denote the linear functional $\rho\circ \Pi$ on $\mathcal{U}(U(n))$. We also define $\Pi_{m, l}$ and $\Pi_{m, \infty}$ as the same as in the previous section. We remark that even if the fundamental representation is \emph{not} generating $\widehat{U(n)}$, the following statement holds true (see also Proposition \ref{prop:subharmonic_vs_extensions}).

\begin{proposition}
    $\Pi_{m, \infty}$ is injective on $\mathbb{M}_m(\mathbb{C})\otimes \mathcal{G}(U(n))$.
\end{proposition}
\begin{proof}
    We suppose that $\Pi_{m, \infty}(e_1)=\Pi_{m, \infty}(e_2)$ for $e_1, e_2\in \mathbb{M}_m(\mathbb{C})\otimes \mathcal{G}(U(n))$. In the representation theory, it is well known that $\widehat{U(n)}$ can be parametrized by $\{\lambda=(\lambda_1\geq\cdots \geq\lambda_n)\in \mathbb{Z}^n\}$. By the Schur--Weyl duality, for any $\lambda$ with $\lambda_n\geq0$ the associated irreducible representation, denoted by $(\Pi_\lambda, \mathcal{H}_\lambda)$, is a sub-representation of $(\Pi^{\otimes l}, (\mathbb{C}^n)^{\otimes l})$ for some $l\geq0$. Thus, we have $(\mathrm{id}_m\otimes \Pi_\lambda)(e_1)=(\mathrm{id}_m\otimes \Pi_\lambda)(e_2)$. If $\lambda_n<0$, let $\lambda'=(\lambda_1-\lambda_n, \dots, \lambda_{n-1}-\lambda_n, 0)$. Then, $(\Pi_\lambda, \mathcal{H}_\lambda)$ is equivalent to the tensor product of $(\Pi_{\lambda'}, \mathcal{H}_{\lambda'})$ and $(\Pi_{(-1, \dots, -1)}^{\otimes -\lambda_n}, \mathcal{H}_{(-1, \dots, -1)}^{\otimes -\lambda_n})$, where $\mathcal{H}_{(-1, \dots, -1)}=\mathbb{C}$ and $\Pi_{(-1, \dots, -1)}$ is given as the inverse of the determinant. Since the second tensor components of $e_1, e_2$ are group-like, we have
    \[(\mathrm{id}_m\otimes \Pi_\lambda)(e_1)=(\mathrm{id}_m\otimes \Pi_{\lambda'}\otimes\Pi_{(-1, \dots, -1)}^{\otimes -\lambda_n})(e_1)=(\mathrm{id}_m\otimes \Pi_{\lambda'}\otimes\Pi_{(-1, \dots, -1)}^{\otimes -\lambda_n})(e_2)=(\mathrm{id}_m\otimes \Pi_\lambda)(e_2).\]
    Therefore, we conclude that $(\mathrm{id}_m\otimes \Pi_\lambda)(e_1)=(\mathrm{id}_m\otimes \Pi_\lambda)(e_2)$ for any $\lambda$, and hence $e_1=e_2$.
\end{proof}

Unlike the general setup, we can show the following:
\begin{proposition}\label{prop:surj}
    $K_{m, \rho}$ coincides with $\overline{\Pi_{m, \infty}(B_{m, \rho})}$.
\end{proposition}
\begin{proof}
    By Proposition \ref{prop:subharmonic_vs_extensions}, $\Pi_{m, \infty}(B_{m, \rho})$ is contained in $K_{m, \rho}$. Combining Theorem \ref{thm:ext_pt1} and Lemma \ref{lem:hilbert}, every element in $\mathrm{exr}(K_{m, \rho})$ can be approximated by $\Pi_{m, \infty}(\mathbb{M}_m(\mathbb{C})_+\otimes \mathcal{E}_\rho(U(N)))$. Thus, by Lemma \ref{lem:choquet_convex1}, $K_{m, \rho}$ is contained in $\overline{\Pi_{m, \infty}(B_{m, \rho})}$.
\end{proof}

Let $\rho\in \mathbb{M}_n(\mathbb{C})^*_+$ be faithful. Based on the previous results, a new understanding of separable elements in $(\mathbb{M}_m(\mathbb{C})\otimes \mathbb{M}_n(\mathbb{C}))_+$ is given as follows:
\begin{theorem}\label{thm:sep_subharmonic}
    The set of all separable elements in $(\mathbb{M}_m(\mathbb{C})\otimes \mathbb{M}_n(\mathbb{C}))_+$ is equal to $\overline{\Pi_{m, 1}(B_{m, \rho})}$.
\end{theorem}
\begin{proof}
    By Corollary \ref{cor:boundary_sep},  $\Pi_{m, 1}(B_{m, \rho})$ is contained the set of separable elements. On the other hand, if $x\in (\mathbb{M}_m(\mathbb{C})\otimes \mathbb{M}_n(\mathbb{C}))_+$ is separable, then there exists a $l$-sub-extension $x_l$ of $x$ for all $l\geq2$. Moreover, we may assume that $(x_l)_{l=0}^\infty\in K_{m, \rho}$, where $x_1:=x$ and $x_0:=(\mathrm{id}_m\otimes\rho)(x)$. Thus, by Proposition \ref{prop:surj}, we have $(x_l)_{l=0}^\infty \in \overline{\Pi_{m, \infty}(B_{m, \infty})}$, and hence $x\in \overline{\Pi_{m, 1}(B_{m, \rho})}$.
\end{proof}

We remark that the quantum de Finetti theorem was not used in the above proof. Thus, we obtain a second proof of Theorem \ref{thm:qD} based on the results in this section.
\begin{proof}[{Second proof of Theorem \ref{thm:qD}}]
    We show only that $x\in (\mathbb{M}_m(\mathbb{C})\otimes \mathbb{M}_n(\mathbb{C}))_+$ is separable if there exists a $l$-sub-extension of $x_l$ of $x$ for all $l\geq2$. We may assume that $(x_l)_{l=0}^\infty\in K_{m, \rho}$, where $x_1=x$ and $x_0:=(\mathrm{id}_m\otimes \rho)(x)$. By Proposition \ref{prop:surj}, we have $(x_l)_{l=0}^\infty \in \overline{\Pi_{m, \infty}(B_{m, \infty})}$, and hence $x$ is separable.
\end{proof}

\appendix

\section{The infinite dimensional setting}
Up to now in this paper, we considered only finite-dimensional $C^*$-algebras. In this appendix, we show how to extend our results to general $C^*$-algebras. Recall that in Section \ref{sec:qdFthm} and \ref{sec:BCD}, we work with the $C^*$-algebra $\mathbb{M}_m(\mathbb{C})$, or work precisely with the dual $\mathbb{M}_m(\mathbb{C})^*$, identified to the algebra by the Hilbert--Schmidt inner product.

\subsection{A generalized quantum de Finetti theorem for general $C^*$-algebras}
We first collect the basic concepts in the theory of $C^*$-algebras we will need and were not introduced in Section \ref{sec:notations}. Our main references will be \cite{Blackadar,BO:book,Pederson}.

Let $\mathcal{C}\subset B(\mathcal{H})$ and $\mathcal{D}\subset B(\mathcal{K})$ be two $C^*$-algebras acting on the Hilbert spaces $\mathcal{H}$ and $\mathcal{K}$. The algebraic tensor product $\mathcal{C}\odot \mathcal{D}$ of $\mathcal{C}$ and $\mathcal{D}$ is naturally a $*$-subalgebra of $B(\mathcal{H}\otimes \mathcal{K})$, and its norm-closure, denoted by $\mathcal{C}\otimes \mathcal{D}$, is the \emph{minimal} tensor product of $\mathcal{C}$ and $\mathcal{D}$.

Let $\mathcal{A}$ and $\mathcal{B}$ be two $C^*$-algebras and $l$ be a positive integer. As in Section \ref{sec:qdFthm}, for any $\sigma\in S_l$, we denote also by $\sigma$ the automorphism of $\mathcal{A}\otimes \mathcal{B}^{\otimes l}$, given for any $a\in \mathcal{A}$, $b_1, \dots, b_l\in \mathcal{B}$ by
\[\sigma(a\otimes b_1\otimes \cdots \otimes b_l):=a\otimes b_{\sigma(1)}\otimes \cdots \otimes b_{\sigma(l)}.\]
By duality, the action of $S_k$ reduces an action on $(\mathcal{A}\otimes B^{\otimes l})^*$. We will denote by $(\mathcal{A}\otimes \mathcal{B}^{\otimes l})^{*, S_l}$ the subspace of $S_l$-invariant elements of $(\mathcal{A}\otimes \mathcal{B}^{\otimes l})^*$, and by $(\mathcal{A}\otimes \mathcal{B}^{\otimes l})^{*, S_l}_+$ the intersection of $(\mathcal{A}\otimes \mathcal{B}^{\otimes l})^{*, S_l}$ and $(\mathcal{A}\otimes \mathcal{B}^{\otimes l})^*_+$.

For any $b\in B_+$, let $\mu^l_b\colon \mathcal{A}\otimes \mathcal{B}^{\otimes l}\to \mathcal{A}\otimes \mathcal{B}^{\otimes (l+1)}$ denote the linear map given by $\mu^l_b(c)=c\otimes b$ for all $c\in \mathcal{A}\otimes \mathcal{B}^{\otimes l}$. By duality, we get a positive linear map $(\mu^l_b)^*\colon (\mathcal{A}\otimes \mathcal{B}^{\otimes (l+1)})^*\to (\mathcal{A}\otimes \mathcal{B}^{\otimes l})^*$ given by $(\mu^l_b)(\psi):=\psi\circ\mu^l_b$ for $\psi\in (\mathcal{A}\otimes \mathcal{B}^{\otimes (l+1)})^*$. Then, $(\mu^l_b)^*((\mathcal{A}\otimes \mathcal{B}^{\otimes (l+1)})^{*, S_{l+1}})\subset (\mathcal{A}\otimes \mathcal{B}^{\otimes l})^{*, S_l}$.

Generalizing the setting of Section \ref{sec:qdFthm} and keeping the above notations, we defined the following convex cone
\[K_{\mathcal{A}^*, b}:=\left\{(\chi_l)_{l=0}^\infty\in \prod_{l=0}^\infty((\mathcal{A}\otimes \mathcal{B}^{\otimes l})^*)^{S_l}_+\,\middle|\,(\mathrm{id}_\mathcal{A}\otimes\mathrm{id}_{\mathcal{B}^{\otimes l}}\otimes b)(\chi_{l+1})\leq \chi_l \text{ for all }l\geq 0\right\}.\]
Endowing $\prod_{l=0}^\infty (\mathcal{A}\otimes \mathcal{B}^l)^*$ with the topology of component-wise weak${}^*$-convergence, $K_{\mathcal{A}^*, b}$ is a closed convex cone.

\begin{lemma}\label{lem:concex_hull_of_extremals}
    If $b \in \mathcal{B}_+$ is invertible, then $K_{\mathcal{A}^*, b}$ is the closed convex hull of $\mathrm{exr}(K_{\mathcal{A}^*, b})$.
\end{lemma}
\begin{proof}
    We only need to prove that for any $r>0$,
    \[C_r:=\{\chi=(\chi_l)_{l=0}^\infty\in K_{\mathcal{A}^*, b}\mid x_0(1_\mathcal{A})\leq r \}\]
    is compact, as the rest of the proof is as in Lemma \ref{lem:choquet_convex1}. Since $b$ is invertible, there exists $\epsilon>0$ such that $b > \epsilon 1_\mathcal{B}$. For every $l\geq 0$, since $\chi_l$ is a positive functional, we have
    \[\|\chi_l\| = \chi_l(1_\mathcal{A}\otimes 1_\mathcal{B}^{\otimes l})\leq \epsilon^{-l}\chi_l(1_\mathcal{A}\otimes b^{\otimes l})\leq \epsilon \chi_0(1_\mathcal{A})\leq \epsilon^{-l}r.\]
    By the Banach--Alaoglu theorem, the ball $B_{l, \epsilon^{-l}r}$ of radius $\epsilon^{-l}r$ in $(\mathcal{A}\otimes \mathcal{B}^{\otimes l})^*$ are compact. Hence, $\prod_{l\geq 0}B_{l, \epsilon^{-l}r}$ is compact, and $C_r\subset \prod_{l\geq 0}B_{l, \epsilon^{-l}r}$ being closed, is also compact.
\end{proof}

By the same proof of Proposition \ref{prop:int_rep}, we obtain the following representation theorem:
\begin{proposition}
    Every $\chi\in K_{\mathcal{A}^*, b}\backslash\{0\}$ can be represented by a Borel probability measure $M_\chi$ on $K_{\mathcal{A}^*, b}$ such that $M_\chi$ is supported on any closed subset of $K_{\mathcal{A}^*, b}$ which contains $\mathrm{exr}(K_{\mathcal{A}^*, b})$.
\end{proposition}

By a similar proof as in Theorem \ref{thm:ext_pt1}, we have
\begin{theorem}\label{thm:A1}
    $\mathrm{exr}(K_{\mathcal{A}^*, b})$ is contained in $\{(\chi\otimes \beta^{\otimes l})_{l=0}^\infty\mid \chi\in \mathcal{A}^*_+, \beta\in \mathcal{B}^*_+\}$.
\end{theorem}

We extend Definition \ref{def:extendable} as follows:
\begin{definition}
    Let $\mathcal{A}$ and $\mathcal{B}$ be $C^*$-algebras. (1) A positive linear form $\psi \in (\mathcal{A}\otimes \mathcal{B})^*_+$ is \emph{separable} if $\psi$ belongs to the closed convex hull of $\mathcal{A}^*_+\otimes \mathcal{B}^*_+$. (2) For $b\in \mathcal{B}_+$ and any $l\geq 1$, an element $\varphi$ of $(\mathcal{A}\otimes \mathcal{B}^{\otimes l})^{*, S_l}_+$ is a \emph{$l$-sub-extension} of $\psi \in (\mathcal{A}\otimes \mathcal{B})^*_+$ with respect to $b$ if $(\mu^l_b)^*(\varphi)\leq \psi$. (3) An element $\psi\in (\mathcal{A}\otimes \mathcal{B})^*_+$ is \emph{$l$-sub-extendable} (with respect to $b$) if there is a $l$-sub-extension.
\end{definition}

Replacing Lemma \ref{lem:choquet_convex1} and Theorem \ref{thm:ext_pt1} by Lemma \ref{lem:concex_hull_of_extremals} and Theorem \ref{thm:A1}, the following theorem is proved as Theorem \ref{thm:qD}.
\begin{theorem}[Generalized quantum de Finetti theorem]
    Let $\mathcal{A}$ and $\mathcal{B}$ be $C^*$-algebras, $b\in \mathcal{B}_+$ be invertible and $\psi \in (\mathcal{A}\otimes \mathcal{B})^*_+$. Then, $\psi$ is separable if and only if $\psi$ is $l$-sub-extendable with respect to $b$ for all $l\geq 1$.
\end{theorem}

\subsection{Infinite-dimensional Biane--Choquet--Deny theorem}
The results that we proved in Section \ref{sec:BCD} can also be extended to a more general $C^*$-algebra context.

Let $G$ be a compact group. We keep the notations introduced in Section \ref{sec:BCD} and fix a positive linear form $\rho\in \mathcal{U}(G)_{*, +}$. For any $m\geq 0$ we defined in Definition \ref{def:4.2} a positive map $P_{m, \rho}$ on $\mathbb{M}_m(\mathbb{C})\otimes \mathcal{U}(G)$.

If $\mathcal{A}$ is a unital $C^*$-algebra, we extend Definition \ref{def:4.2} to a map $P_{\mathcal{A}^*, \rho}$ defined on the tensor product (as linear spaces) of $\mathcal{A}^*\otimes \mathcal{U}(G)$ by $P_{\mathcal{A}^*, \rho}:=\mathrm{id}_{\mathcal{A}^*}\otimes P_\rho$.

We say that $x\in \mathcal{A}^*\otimes \mathcal{U}(G)$ is positive if $(\mathrm{id}_{\mathcal{A}^*}\otimes \pi)(x)$ is a positive linear form in $\mathcal{A}^*\otimes B(\mathcal{H})=(\mathcal{A}\otimes B(\mathcal{H}))^*$ for any finite-dimensional unitary representation $(\pi, \mathcal{H})$ of $G$. Then, by extending the proof of Proposition \ref{prop:cp}, we can check that $P_{\mathcal{A}^*, \rho}$ is a positive map.

Extending the definition given in Equation \eqref{eq:subharmonic}, we consider the following convex cone
\[B_{\mathcal{A}^*, \rho}:=\{x\in (\mathcal{A}^*\otimes \mathcal{U}(G))_+\mid P_{\mathcal{A}^*, \rho}(x)\leq x\}.\]

To generalize the statements and the proofs of Lemma \ref{lem:choquet_convex2}, Proposition \ref{prop:int_rep2}, and Theorem \ref{thm:BCD}, recall (see Section \ref{sec:alg_cpt_groups}) that if $((\pi_\alpha, \mathcal{H}_\alpha))_{\alpha\in \widehat G}$ is the set of irreducible representations of $G$, then each $\pi_\alpha$ extends to a unital $*$-homomorphism from $\mathcal{U}(G)$ to $B(\mathcal{H}_\alpha)$, and that $\mathcal{U}(G)$ is $*$-isomorphic to $\prod_{\alpha\in \widehat G} B(\mathcal{H}_\alpha)$ by $\Pi_G(x):=(\pi_\alpha(x))_{\alpha\in \widehat G}$. Therefore, $\Pi_{\mathcal{A}^*, G}:=\mathrm{id}_{\mathcal{A}^*}\otimes \Pi_G$ is an linear isomorphism from $\mathcal{A}^*\otimes \mathcal{U}(G)$ onto $\mathcal{A}^*\otimes \prod_{\alpha\in \widehat G} B(\mathcal{H}_\alpha)\subset \prod_{\alpha\in \widehat G}\mathcal{A}^*\otimes B(\mathcal{H}_\alpha)$.

For any $\alpha\in \widehat G$, $H_\alpha$ being finite dimensional, we identify $\mathcal{A}^*\otimes B(\mathcal{H}_\alpha)$ with $(\mathcal{A}\otimes B(\mathcal{H}_\alpha))^*$, and we endow $\prod_{\alpha\in\widehat G}\mathcal{A}^*\otimes B(\mathcal{H}_\alpha)$  with the product topology of $(\mathcal{A}\otimes B(\mathcal{H}_\alpha))^*$, endowed with the wak${}^*$-topology. By considering the topology on $\mathcal{A}^*\otimes \mathcal{U}(G)\subset (\mathcal{A}\otimes \mathbb{C}[G])^*$ given by the weak${}^*$ topology on $(\mathcal{A}\otimes \mathbb{C}[G])^*$, the $\Pi_{\mathcal{A}^*, G}$ is a topological isomorphism.

Generalizing Lemma \ref{lem:choquet_convex2} Proposition \ref{prop:int_rep2}, and their proof, we have
\begin{proposition}
    $B_{\mathcal{A}^*, \rho}$ is the closed convex hull of $\mathrm{exr}(B_{\mathcal{A}^*, \rho})$. Moreover, every $x\in B_{A^*, \rho}\backslash\{0\}$ is represented by a Borel probability measure on $\mathrm{exr}(B_{\mathcal{A}^*, \rho})$ (in the sense of Prop. \ref{prop:int_rep}).
\end{proposition}
\begin{proof}
    For any $\alpha\in \widehat G$ we define a linear functional $1_A\otimes \mathrm{Tr}_\alpha$ on $\mathcal{A}^*\otimes B(\mathcal{H}_\alpha)$ by
    \[(1_\mathcal{A}\otimes \mathrm{Tr}_\alpha)(y):=[(\mathrm{id}_\mathcal{A}\otimes \mathrm{Tr}_\alpha)(y)](1_\mathcal{A}) \quad (y\in \mathcal{A}^*\otimes B(\mathcal{H}_\alpha)),\]
    where $\mathrm{Tr}_\alpha$ is the (non-normalized) trace on $B(\mathcal{H}_\alpha)$. Moreover, for every $c=(c_\alpha)_{\alpha\in \widehat G}\in \mathbb{R}_{>0}^{\widehat G}$ we define $\varphi_c\colon (\mathcal{A}^*\otimes \mathcal{U}(G))_+\to [0, \infty]$ by
    \[\varphi_c(x):=\sum_{\alpha \in \widehat G}c_\alpha (1_A\otimes \mathrm{Tr}_\alpha)(\Pi_{\mathcal{A}^*, \alpha}(x)) \quad (x\in (\mathcal{A}^*\otimes \mathcal{U}(G))).\]
    To prove the statement, it suffices to show that $C_{\varphi_c}:=\{x\in B_{\mathcal{A}^*, \rho}\mid \varphi_c(x)\leq 1\}$ is compact. In fact, the remaining is the same as in Section \ref{sec:BCD}. If $x\in C_{\varphi_c}$, then for every $\alpha\in \widehat G$, we have
    \[\|\Pi_{\mathcal{A}^*, \alpha}(x)\|=|[\Pi_{\mathcal{A}^*, \alpha}(x)](1_\mathcal{A}\otimes 1_{B(\mathcal{H}_\alpha)})|=|(1_\mathcal{A}\otimes \mathrm{Tr}_\alpha)(\Pi_{A^*, \alpha}(x))|\leq c_\alpha^{-1}.\]
    Thus, we have $C_{\varphi_c}\subset \{x\in \mathcal{A}^*\otimes \mathcal{U}(G)\mid \|\Pi_{\mathcal{A}^*, \alpha}(x)\|\leq c_\alpha^{-1}\}$, and hence $C_{\varphi_c}$ is compact.
\end{proof}

By assuming as in Section \ref{sec:BCD} that $\rho$ is generating, and by a proof similar to Theorem \ref{thm:BCD}, we obtain the following result.
\begin{theorem}[{$\mathcal{A}^*$-valued Biane--Choquet--Deny theorem}]
    The extreme points $\mathrm{exr}(B_{\mathcal{A}^*, \rho})$ are contained in $\mathcal{A}^*_+\otimes \mathcal{E}_\rho(G)$. In particular, $B_{\mathcal{A}^*, \rho}$ is the closed convex hull of $\mathcal{A}^*_+\otimes \mathcal{E}_\rho(G)$.
\end{theorem}
}


\begin{thebibliography}{99}
    \bibitem{Biane92} P. Biane, \'Equation de Choquet--Deny sur le dual d'un groupe compact, \emph{Probab. Theory Related Fields}, \textbf{92}(1) (1992), 39--51.
    \bibitem{Biane92-2} P. Biane, Minuscule weights and random walks on lattices, \emph{Quantum Probability and Related Topics}, \textbf{7}(1992), 51--65.
    \bibitem{Biane08} P. Biane, Introduction to random walks on noncommutative spaces, \emph{Quantum Potential Theory}, 61--116, Springer, Berlin, 2008.
    \bibitem{Blackadar} B. Blackadar, \emph{Operator Algebras: Theory of $C^*$-Algebras and von Neumann Algebras}, Encyclopedia of Mathematical Science \textbf{122}, Springer-Verlag, Berlin Heidelberg, 2006.
    \bibitem{BO:book} N. P. Brown, N. Ozawa, \emph{$C^*$-Algebras and Finite-Dimensional Approximations}, Graduate Studies in Mathematics, {\bf 88}, American Mathematical Society, Providence, RI.
    \bibitem{Izumi} M. Izumi, Non-commutative Poisson boundaries and compact quantum group actions, \emph{Adv. Math.} 169 (2002), no. 1, 1–57.
    \bibitem{KoenigRenner} R. Koenig, R. Renner, A de Finetti representation for finite symmetric quantum states, \emph{J. Math. Phys. 46}, 122108 (2005)
    \bibitem{KS} J. Krajczok, P. M. So\l tan, Center of the algebra of functions of the quantum group $SU_q(2)$ and related topics, \emph{Commentationes Mathematicae} \textbf{56}(2) (2016), 251--272.
    \bibitem{NT:book} S. Neshveyev, L. Tuset, \emph{Compact Quantum Groups and Their Representation Categories}, Soc. Math. France, 2013.
    \bibitem{Pederson} G. K. Pederson, \emph{$C^*$-Algebras and Their Automorphism Groups}, second edition, S. Eilers, D. Olesen (Eds.), Pure and Applied Mathematics, Academic Press, London, 2018.
    \bibitem{Phelps:book} R. R. Phelps, \emph{Lectures on Choquet's theorem}, Lecture Notes in Mathematics \textbf{1757}, Springer Berlin, Heidelberg, 2001.
    \bibitem{Stormer} E. St\o rmer, Symmetric States of Infinite Tensor Products of C*-algebras, \emph{J. Funct. Anal.} \textbf{3} (1969), 48--68. 
    
    \bibitem{Woess} W. Woess, \emph{Random Walks on Infinite Graphs and Groups}, Cambridge Tracts in Mathematics \textbf{138}, Cambridge University Press, 2000.
\end{thebibliography}
\end{document}